\documentclass[10pt]{article}

\usepackage[utf8]{inputenc}
\usepackage{amsfonts,amsthm,amssymb,amsmath}
\usepackage{xcolor}
\usepackage{geometry}
\usepackage{cite}

\newcommand{\R}{\mathbb{R}}
\newcommand{\rd}{\mathrm{d}}
\newcommand{\eps}{\varepsilon}
\newcommand{\ini}{{(\mathrm{I})}}
\newcommand{\e}{\mathrm{e}}
\newcommand{\bydef}{:=}

\newcommand{\bm}{\mathcal{BM}}
\newcommand{\bmp}{\mathcal{BM}^+}
\newcommand{\ueps}{u^{(\eps)}}
\newcommand{\veps}{v^{(\eps)}}
\newcommand{\D}{\mathcal{D}}
\newcommand{\Df}{\mathfrak{D}}
\newcommand{\Ep}{\mathcal{E}'}
\newcommand{\HH}{\mathcal{H}}

\newcommand{\ddt}{\frac{\mathrm{d}}{\mathrm{d}t}}

\newtheorem{thm}{Theorem}
\newtheorem{lem}{Lemma}
\newtheorem{prop}{Proposition}
\newtheorem{rmk}{Remark}
\newtheorem{defi}{Definition}

\title{Conservative parabolic problems:  non-degenerated theory and degenerated examples from population dynamics}

\author{%
Olga Danilkina\thanks{%
Department of Mathematics, 
College of Natural and Mathematical Sciences, 
the University of Dodoma, P.O. BOX 259, Dodoma, Tanzania.
olga.danilkina@gmail.com},
 Max O. Souza\thanks{%
 Departamento de Matem\'atica Aplicada,
 Universidade Federal Fluminense,
 R. M\'ario Santos Braga, s/n,
 Niter\'oi, RJ 24020-140, Brasil.
 maxsouza@id.uff.br
 },
  Fabio A. C. C. Chalub\thanks{Departamento de Matem\'atica 
and Centro de Matem\'atica e Aplica\c c\~oes, 
Universidade Nova de Lisboa,
Quinta da Torre, 2829-516, Caparica, Portugal.
chalub@fct.unl.pt}}

\begin{document}

\maketitle

\begin{abstract}
We consider partial differential equations (PDE) of drift-diffusion type in the unit interval, supplemented by either two conservation laws or  by a conservation law and a further boundary condition. We treat two different cases: (i) uniform parabolic problems; (ii) degenerated problems at the boundaries. The former can be treated in a very general and complete way, much as the traditional boundary value problems. The latter,  however, bring new issues, and we restrict our study to  a class of forward Kolmogorov equations that arise naturally when the corresponding stochastic process has either one or two absorbing boundaries. These equations are treated by means of a uniform parabolic regularisation, which then yields a measure solution in the vanishing regularisation limit that is unique. Two prototypical problems from population dynamics are treated in detail.
\end{abstract}

\section{Introduction}
\label{sec:intro}

\subsection{Background}

Partial differential equations (PDE) are an ubiquitous  tool for modelling a variety of phenomena in the applied sciences. Typically, a PDE in any given domain will not be well-posed, at least in the sense of uniqueness,  unless further assumptions are made on the solutions. Such assumptions can range from  traditional  boundary conditions in bounded domains --- eg. Dirichlet or  Neumann --- to  integrability conditions in unbounded domains, and might include a mixture of both --- as for instance, in the case of degenerated equations in part of the boundary. This is the traditional state of affairs as presented in most of the classical introductions to the subject \cite{John1982,Taylor96a,Evans98,Friedman2013}.

In a different perspective, a number of these models are derived from integral formulations that arise naturally from conservation laws that are expected to hold in the problem being modelled. This is well known in the literature of conservation laws in hyperbolic problems --- cf. \cite{Dafermos2005} --- but it also appear in other settings within hyperbolic problems: \cite{Pulkina1992,Bouziani:Benouar:1996,Pulkina2007}. In the context of parabolic operators, however, this class of problems seems to have received  much less attention.

A very similar class of problems, though, has received somewhat more attention: the solution of parabolic operators with the specification of  an integral constraint  and a boundary condition. The study of these problems seems to date back at least to \cite{Cannon1963,Deckert:Maple:1963}, who studied specific problems in this class for the heat equation. A proof of existence and uniqueness for  a problem in this class, with a general linear parabolic operator  can be found already in \cite{Kamynin1964}.  The subject has resurfaced from the late seventies to the early nineties with the works of  \cite{Ionkin1977,Ionkin1979,Cannon_Hoek,Cannon1986,Yurchuk1986,Shi1993}.

Beginning in the late nineties, there is a growing  interest in understanding the solutions of the heat equation subject to the specification of the first two linear moments, which seems to be first addressed by \cite{Bouziani1996}, but also include \cite{Bouziani1997,Bouziani2002,Dai2007}. Numerical methods for these problems were also developed in a series of works~\cite{Dehghan2003,Bouziani2008}. More recently, using a combination of variational and semigroup methods \cite{Mugnolo:Nicaise:2014} presents a very detailed theory for the heat and wave equation in this setting. See also \cite{Mugnolo:Nicaise:2014b} for companion results to diffusive equation with the $p$-Laplacian .  Further discussion and references can be found in \cite{Mugnolo:Nicaise:2014} and in the references therein.

As far as conservation laws for parabolic problems are concerned, the earliest work that we are aware of  is due to the two last authors in \cite{ChalubSouza09a}, which show that a degenerated parabolic equation --- the so-called generalised Kimura equation from population genetics --- subject to two conservation laws --- namely the conservation of probability and conservation of centering with respect to the fixation probability --- is well posed in the space of Radon measures. A generalisation of this problem to higher dimensions in the context of the Wright-Fisher process is given in \cite{ChalubSouza14}, and a study of the  PDE version for the SIS epidemiological model which is degenerated  at the origin and needs a boundary condition at 1 is given in \cite{ChalubSouza14b}. A recent work showing that the heat equation subject to conservation of the first two moments  is well posed in all $L^p$ spaces for $p\geq 1$, and even in $C^0$ is \cite{Bobrowski:Mugnolo:2013}.

\subsection{Degenerated problems from mathematical biology}

The ultimate goal  of this work is to understand how the solution of   parabolic conservative problems, which are degenerated at least in one of the boundaries,   can be approximated by the solutions of  non-degenerated problems. As a by-product, we will be able to characterise a condition the guarantees positiveness of the solution obtained. It will also  allow to obtain small  test function spaces where existence and uniqueness of the solution hold.

In what follows, we shall  focus in two examples. Our first example is the so called Generalised Kimura Equation~\cite{ChalubSouza09a}:
\begin{equation}\label{eq:replicator_diffusion}
 \partial_tu=\partial_x^2\left(x(1-x)u\right)-\partial_x\left(x(1-x)\psi(x)u\right)\ ,
\end{equation}
where $(x,t)\in[0,1]\times\R_+$ and $\psi:[0,1]\to\R$ is a  function in a space to be defined latter on (known as fitness in the biological literature).
The initial condition is given by 
\begin{equation}\label{eq:initial_condition}
 u^\ini=u(\cdot,0)\ge 0\ ,\quad x\in[0,1]\ ,
\end{equation}
and the solution of this equation should obey two integral conditions (conservation laws):
\begin{align}
\label{eq:conserv_law_1}
&\frac{\rd\ }{\rd t}\int_0^1 u\rd x=0\ ,\\
\label{eq:conserv_law_2}
&\frac{\rd\ }{\rd t}\int_0^1 \varphi(x)u\rd x=0\ ,
\end{align}
where $\varphi:[0,1]\to\R_+$ is the unique solution of 
\begin{equation}\label{fix_prob}
\varphi''+\psi(x)\varphi'=0,\ \text{with}\ \varphi(0)=0\ \text{and}\ \varphi(1)=1\ .
\end{equation}
Therefore, 
\begin{equation}\label{varphi_sol}
 \varphi(x)=\frac{\int_0^x\e^{-\int_0^y\psi(z)\rd z}\rd y}{\int_0^1\e^{-\int_0^y\psi(z)\rd z}\rd y}\ .
\end{equation}

The derivation of the equation above with the conservation laws was performed in~\cite{ChalubSouza09a}. It was extensively studied in work~\cite{ChalubSouza09b}. Generalisation for more dimensions can be found in~\cite{ChalubSouza14}. 
Different systems, but with similar characteristics were also studied in~\cite{ChalubSouza11,ChalubSouza14b}.
See also~\cite{EpsteinMazzeo2013}. A similar equation is studied in~\cite{Yang2014}, where mutations were allowed. In this last reference, as boundaries were not absorbing, boundary conditions were imposed and classical solutions were found. 

Our second example comes from epidemiology: the SIS-PDE model that is given by
\begin{equation}\label{sis_pde}
\partial_tp=-\partial_x\left\{x\left[R_0(1-x)-1\right]p\right\}
+\frac{1}{2}\partial_x^2\left\{x(R_0(1-x)+1)p\right\},
\end{equation}
satisfying the boundary condition
\begin{equation}\label{sis_pde_bc}
\left[ \frac{1}{2}\left[(1-R_0)p|_{1}+\partial_xp|_1\right]+p|_{1}\right]=0
\end{equation}
and the conservation law
\begin{equation}
\label{eqn:cons_law}
\frac{\rd}{\rd t}\int_0^1p(x,t)\,\rd x =0.
\end{equation}
This model is an intermediate model in between the Markov process associated to the \textbf{S}usceptible-\textbf{I}nfecious-\textbf{S}uscptible epidemiological model and its ODE counterpart. See~\cite{ChalubSouza14b} for further details.

Note that equations~(\ref{eq:replicator_diffusion}) and~(\ref{sis_pde}) are degenerated Fokker-Planck equations, for which no boundary condition can be imposed in the degenerated boundary, other than integrability. These equations are associated to diffusion approximation of Markov chains with absorbing states, and are studied, for instance, in~\cite{DiBenedetto93}; applications to biology in a framework similar to ours can be found in~\cite{Feller1951,Ewens,Ethier1976,EpsteinMazzeo2013}.

\subsection{Summary of results and outline}

We begin in Section~2 by defining a non-degenerated conservative parabolic problem, and the first result  shows that the conservative problem can be converted to a coupled boundary value problem, and that the conservations laws must be  related to the kernel of the formal adjoint. The next step is then to define non-degenerated  positive problems in terms of the conservation laws. For a very general  class of positive problems, we are able to show existence and uniqueness along the same lines of the more classical boundary conditions, by recurring to general Sturm-Liouville theory. In particular, if the coefficients of the equations are smooth, we obtain that the solutions are $C^\infty$. This Section can be seen as an extension of the theory presented in \cite{Mugnolo:Nicaise:2014}.

In Section 3 we start the study of our degenerate problems. We introduce a class of self-adjoint elliptic perturbations. These perturbed problems are then readily amenable to be treated by theory developed in the previous Section. It turns out that the solutions can be naturally treated as measures, and Prokohorov theorem can be used to pass the limit as the perturbation vanishes. This very weak solution then satisfies all the required conditions. A further analysis of the solution is presented in Section 4, where we obtain, depending on the regularity of the coefficients, several representation forms of the solution. In particular, we show that the solution can be written as the sum of two atomic measures at the boundaries and the classical solution.  Also, optimal domains of test functions for the weak solution are also discussed.

We close this work with a discussion of the results presented in Section~5.

\bigskip

\section{Conservative parabolic problems}
\label{sec:reg_cpp}
We begin by considering uniformly  parabolic problems in self-adjoint form. To this end, let
\[
Lv=\partial_x\left(p(x)\partial_xv\right)+q(x)v
\]
defined in some finite interval $I=(a,b)$, and with $p>0$, and with $p,q\in L^2(I)\cap L^\infty(I)$.

\begin{defi}
A (totally) conservative parabolic problem is an initial value problem of the form:
\begin{equation}
\label{eq:def_tcp}
\left\{
\begin{array}{lr}
\partial_tv=Lv,& \text{in }I\text{ and }t>0;\\
\ddt\langle v(\cdot,t),\phi_1\rangle=0,& t>0;\\
\ddt\langle v(\cdot,t),\phi_2\rangle=0,& t>0;\\
v(x,0)=v_0(x).&
\end{array}
\right.
\end{equation}
where $\langle\cdot,\cdot\rangle$ denotes the inner product in $L^2(I)$ and $\phi_1,\phi_2\in L^2(I)$ are not multiple of one another.
\end{defi}

If the infinitesimal generator  associated to  Equation~\eqref{eq:def_tcp} is self-adjoint, then it turns out that the possible choices for $\phi_1$ and $\phi_2$ are limited:

\begin{thm}
	\label{thm:eqv_cpp}
	The operator $L$ in \eqref{eq:def_tcp} can be taken to be self-adjoint if, and only if, $\phi_1$ and $\phi_2$ are linearly independent solutions of the ODE $Lv=0$, and if problem \eqref{eq:def_tcp} can be recast as a coupled (non-local) boundary value problem as follows
	\begin{equation}
	\label{eq:def_tcp_m}
	\left\{
	\begin{array}{lr}
	\partial_tv=Lv,& \text{in }I\text{ and }t>0;\\
	p(b)\left[\partial_xv(b,t)\phi_1(b)-v(b,t)\phi_1'(b)\right]-p(a)\left[\partial_xv(a,t)\phi_1(a)-v(a,t)\phi_1'(a)\right]=0,& t>0;\\
	p(b)\left[\partial_xv(b,t)\phi_2(b)-v(b,t)\phi_2'(b)\right]-p(a)\left[\partial_xv(a,t)\phi_2(a)-v(a,t)\phi_2'(a)\right]=0,& t>0;\\
	v(x,0)=v_0(x).&
	\end{array}
	\right.
	\end{equation}
	Furthermore, we have that  $\phi_1$ and $\phi_2$ are eigenfunctions of $L$ associated to the eigenvalue $\lambda=0$.
\end{thm}	
	
	\begin{proof}
		Assume $L$ is self-adjoint with domain $D(L)$. Then the conservation conditions are equivalent to
		\[
		\langle \varphi,L\phi_i\rangle=0,\quad i=1,2;\qquad \varphi\in D(L).
		\]
		Since $L$ is self-adjoint, it is densely-defined, and hence the identities above hold on $L^2(I)$. Therefore, we have that $L\phi_i=0$, i=1,2. Since they are not multiple of one another, they are linearly independent solutions of $Lv=0$. Also, under the assumptions on the coefficients of $L$, we have that any solution to the ODE is of class $C^1$. Thus, since $L$ is self-adjoint, direct integration by parts yields \eqref{eq:def_tcp_m}.
		
		Conversely, assume that  problem~\eqref{eq:def_tcp} is equivalent to problem~\eqref{eq:def_tcp_m}. Then a direct computation shows that  $L$ is  symmetric. Since $\phi_1,\phi_2$ are solutions of $Lv=0$,  it  follows immediately from Proposition~2.1 in \cite{Bailey:etal:1996}---see also \cite{Weidmann:1987,Kong1996}---that $L$ can be taken to be self-adjoint.
		
		For the last claim,  we already have that $L\phi_i=0$, for $i=1,2$. Therefore,  it remains  to show  that every $\phi_i$ satisfies the boundary conditions. We shall check this for $\phi_1$ --- the case of $\phi_2$ is analogous.  The first condition is easily verified by direct substitution of $\phi_1$ into $v$. In order to verify the  second condition, we write $B_2$ for its left hand side and compute:
		\begin{align*}
		B_2&=p(b)\left[\phi_1'(b) \phi_2(b)-\phi_1(b)\phi_2'(b)\right]-p(a)\left[\phi_1'(a)\phi_2(a)-\phi_1(a)\phi_2'(a)\right]\\
		&=p(b)W_{\{\phi_1,\phi_2\}}(b)-p(a)W_{\{\phi_1,\phi_2\}}(a)=0,
		\end{align*}
		where the last equality follows from Abel's theorem.
	\end{proof}

\begin{defi}
We say that \eqref{eq:def_tcp_m} is a non-negative conservative problem, if we have a conservative problem with $\phi_1$ and $\phi_2$ being non-negative.  If, in addition, there exists at least one solution  of the ODE $Lv=0$ that is positive everywhere then we say that the problem is intrinsically positive.
\end{defi}

\begin{lem}
Assume that  we have an intrinsically  positive conservative problem, and consider the  corresponding spectral problem
\[
-Lu=\lambda u,
\]
with the coupled boundary conditions.

Then,  we can order the eigenvalues such that we have 
\[
0=\lambda_1=\lambda_2<\lambda_3\leq \cdots \lambda_n\leq \lambda_{n+1}\leq\cdots.
\]
%
\end{lem}
\begin{proof}
Since $p>0$, we have that the set of  eigenvalues are bounded from below and and is unbounded from above. In particular, we can order the eigenvalues such that 
\[
\lambda_1\leq\lambda_2\leq\lambda_3\leq \cdots \lambda_n\leq \lambda_{n+1}\leq\cdots,
\]
From Theorem~\ref{thm:eqv_cpp}, we already know that zero is an eigenvalue of multiplicity two.   Since the problem is intrinsically positive, we have a choice of two linear independent eigenfunctions that are positive everywhere. From the oscillatory theory of Sturm-Liouville problems \cite[Theorem 13.5]{Weidmann:1987}, if we have  $\lambda_n=\lambda_{n+1}$ then the corresponding eigenfunctions have either $n-1$ or $n$ zeros. Thus, if $\lambda=0$ is not the principal eigenvalue, any eigenfunction associated to it must have at least one zero in the interior, and this is not possible for an intrinsically positive problem. Hence $\lambda_1=\lambda_2=0$. Since the problem is second order, the eigenvalues can have multiplicity at most two, and hence $\lambda_3>0$.
\end{proof}
\begin{rmk}
\label{rmk:reg_evp}
Recall  that if $p,q\in L^2(I)$ then the eigenfunctions are in $H^2(I)$. In general, if $p,q\in C^k(I)$, then the  eigenfunctions will be of class $C^{k+2}$---cf. \cite{CL,Zettl2005}.
\end{rmk}
In what follows, we shall assume that $p$ and $q$ are given, and for an intrinsically positive problem we shall write $w_1=\psi_1$, $w_2=\psi_2$, where $\psi_1,\psi_2$ are eigenfunctions corresponding to the zero eigenvalue, with unit norm, and that are positive everywhere.  Furthermore, the set  $\{w_k,k\geq 3\}$ will correspond to the unit  eigenfunctions associated to the positive eigenvalues.

Before we state the next result, we  introduce the following harmonic spaces --- cf. \cite{Taylor96a}: for non-negative $s$, we define
\[
\HH_s=\left\{f \in L^1(I) \text{ s.t. } \sum_{k=0}^\infty \hat{f}(k)\lambda_k^{s/2}w_j  \in L^2(I)\right\},\quad \hat{f}(k)=\langle f,w_k)\rangle ,
\]
for $s\in\R$. Notice that
\[
\|f\|_s^2=\sum_{k=1}^\infty |\hat{f}(k)|^2\lambda_k^s,
\]
 and that we have $\HH_0=L ^2(I)$ and  $\HH_1=H ^1(I)$. In addition, we have that $\HH_{-s}$ is the dual of $\HH_s$, and that the Fourier series characterisation is still valid.

\begin{thm}
\label{thm:eup}
Consider the initial value problem \eqref{eq:def_tcp_m} and assume that it is intrinsically positive.  If we have  $v_0\in L^2(I)$, then, for any $T>0$,  there exists a unique solution in the class $C([0,T);L^2(I))\cap C^\infty((0,T);\HH_1)$. Furthermore, if $v_0\geq0$, then we have  $v(\cdot,t)\geq0$ for every $t$.
\end{thm}
\begin{proof}
The existence proof is standard, and we have that
\[
v(x,t)=\sum_{k=1}^\infty a_k\e^{-\lambda_kt}w_k,\quad a_k=\langle v_0,w_k\rangle.
\]
Before we  discuss uniqueness, we shall investigate the positiveness of the solutions.
In order to show the positiveness of a given solution at any time, we first assume that $v_0>0$, and that $p$ and $q$ are smooth. In this case we have that  $v \in C([0,T);L^2(I))\cap C^\infty((0,T);C^{\infty}(I))$.

Let
\[
v_s=a_1w_1+a_2w_2 \text{ and } v_t=\e^{-\lambda_3t}\sum_{k=3}^\infty a_k\e^{-(\lambda_k-\lambda_3)t}w_k
\]
be the steady and transient parts of the solution.

Since the problem is positive, we have that $w_1,w_2>0$, and thus if $v_0>0$ we have $a_1,a_2>0$.

We also observe that
\[
\lim_{t\to\infty}\e^{\frac{\lambda_3}{2}t}v_t=0.
\]
Thence, there exists a time $T$ such that $v(\cdot,t)>0$, for $t\geq T$. Let $t^*\geq0$ be the minimal time such that $v(.,t)\geq0$, for $t\geq t^*$. Clearly $0\leq t^*<T$. 

Assume that $t^*>0$. Since $v$ is smooth,  we have that
\begin{equation}
\label{eqn:lim_ts}
\lim_{t\downarrow t^*}\inf_{\mathrm{int(I)}}{v(\cdot,t)}=0
\end{equation}
On the other hand, since the coefficients are bounded, we can assume without loss of generality that the strong maximum principle holds for $L$. Thus, the parabolic Harnack inequality---cf. \cite{Evans98,Lieberman92}---holds for $t\geq t^*$, and together with  Equation~\ref{eqn:lim_ts} yields that $v(\cdot,t^*)=0$. 

Let $\phi_1$ be a positive conservation law, which exists since the problem is intrinsically positive. Then we have
\[
0=\langle v(\cdot,t^*),\phi_1\rangle=\langle v(\cdot,T),\phi_1\rangle >0,
\]
which is a contradiction. Therefore $t^*=0$, and the solution is positive.

If $v_0\geq0$,  add a positive constant  $K$ to $v_0$, and then by  letting  $K\to0$ we obtain that $v(\cdot,t)\geq0$ as claimed.

Finally, the result for $p,q\in L^2(I)\cap L^\infty(I)$ follows by a standard mollification argument.

For uniqueness, we first observe that, if $v_0\equiv0$,  then we must have  $v(\cdot,t)\geq0$, for all time.   As before, let  $\phi_1$ be a positive conservation law  for \eqref{eq:def_tcp}. Then
\[
\langle v(\cdot,t),\phi_1\rangle = \langle v_0,\phi_1\rangle=0.
\]
Therefore, we have $v(\cdot,t){\color{yellow}\equiv}{\color{blue}=}0$, and uniqueness follows.
\end{proof}

We now want to briefly discuss the case when there is only one conservation law, namely:
\begin{defi}
A partially conservative problem is an initial value given by
\begin{equation}
\label{eq:pcp}
\left\{
\begin{array}{lr}
\partial_tv=Lv,& \text{in }I\text{ and }t>0;\\
\ddt\langle v(\cdot,t),\phi_1\rangle=0,& t>0;\\
B(v(a,t),v(b,t),\partial_x v(a,t),\partial_x v(b,t))=0,&\\
v(x,0)=v_0(x),&
\end{array}
\right.
\end{equation}
if  $L\phi_1=0$, and the corresponding spectral problem is self-adjoint. We shall also say that  \eqref{eq:pcp} is a positive problem, if $\phi_1>0$.
\end{defi}

Minor modifications of the previous arguments are necessary to prove the following:
\begin{thm}
\label{thm:pcp_eup}
Consider the initial value problem \eqref{eq:pcp} and assume that it is positive partially conservative problem.  If we have  $v_0\in L^2(I)$, then, for any $T>0$,  there exists a unique solution in the class $C([0,T);L^2(I))\cap C^\infty((0,T);\HH_1)$.  If, in addition, $\partial_tv-Lv$ satisfies the strong maximum principle, then for  $v_0\geq0$ we have that  $v(\cdot,t)\geq0$ for every $t$.
\end{thm}

\begin{rmk}
The heat equation with homogeneous Neumann conditions is a positive partially conservative problem. Indeed, one can take $\phi_1=1$, and as an additional boundary condition that the flux in one of the endpoints should vanish. In this case, while the existence and uniqueness are well-known, Theorem~\ref{thm:pcp_eup}   provides which seems to be  a new alternative argument for positiveness that does not require knowledge of the boundary behaviour of the solution.
\end{rmk}

\begin{rmk}
For the general problem
\[
\left\{
\begin{array}{lr}
\partial_tu=Mu, t>0,\, x\in (a,b)\\
\ddt (u,\phi_1)=0,& t>0\\
\ddt (u,\phi_2)=0,& t>0\\
u(0,x)=u_0(x).&
\end{array}
\right.
\]
with
\[
Mu=a(x)\partial_x^2u+b(x)\partial_xu+c(x)u,
\]
and with the coefficient $a$ bounded away from zero, we set 
\[
 \eta(x)=\exp\left(\int_a^x\frac{b(s)}{a(s)}\,\rd s\right).
\]
Then, we can recast the problem as 
\[
\left\{
\begin{array}{lr}
\partial_tu=Lu, t>0,\, x\in (a,b)\\
\ddt (u,\psi_1)=0,& t>0\\
\ddt (u,\psi_2)=0,& t>0\\
u(0,x)=u_0(x).&
\end{array}
\right.
\]
with
\[
Lu=\frac{a}{\eta}\left[\partial_x\left(\eta\partial_x u\right)+\frac{c\eta}{a}u\right]
\]
and
\[
\psi_i=\frac{a}{\eta}\phi_i,\quad i=1,2.
\]
Then all previous results apply, provided all the spaces are weighted with respect to the measure
\[
\rd\mu=\frac{\eta}{a}\rd x.
\]
Notice also  that all the inner products, inclusive the ones in the conservation laws, are now given with respect to weighted inner product.
\end{rmk}

\begin{rmk}
Naturally, if one prescribe the linear moments to be constant, and compatible with the initial conditions, one gets a conservative parabolic problem. On the other way round,
by considering, without loss of generality, that $\phi_1$ and $\phi_2$ have unit norm and are orthogonal, if we specify
\[
\langle v(\cdot,t),\phi_i\rangle=F_i(t),\quad i=1,2.
\]
We can write 
\[
w(x,t)=v(x,t)-F_1(t)\phi_1(x)-F_2(t)\phi_2(x)
\]
Then $w$ satisfies a non-homogeneous conservative parabolic problem. Indeed, in this case we have
\[
\langle w(\cdot,t),\phi_i\rangle=0,\quad i=1,2.
\]
and
\[
\partial_t w=Lw+G(x,t),\quad G(x,t)=F'_1(t)\phi_1(x) + F'_2(t)\phi_2(x).
\]
This last problem can then be solved using the  Duhamel principle.
\end{rmk}

\section{From degenerated to non-degenerated problems and back}

We are interested in deal with Fokker-Planck equations of the following type
\begin{equation}
\partial_tu=\partial^2_x\left(gu\right)-\partial_x\left(g\psi u\right)=0
\label{eqn:gen_fp}
\end{equation}
where we shall always assume that $g\in C^\infty([0,1])$, that $g(0)=0$ and that $u$ satisfies
\begin{equation}
\frac{\rd}{\rd t}\int_0^1u(x,t)\,\rd x=0
\label{eqn:cons_law_prob_genfp}
\end{equation}

We shall interested in two different cases:
\begin{enumerate}
\item $g(1)=0$, and $u$ then further satisfies
\begin{equation}
\frac{\rd}{\rd t}\int_0^1\varphi(x)u(x,t)\,\rd x=0
\label{eqn:cons_law_fix_genfp}
\end{equation}
\item $g(1)>0$, and $u$ then further satisfies
\begin{equation}
\left.\partial_x(gu)-g\psi u\right|_{x=1}=0
\label{eqn:bnd_cond_genfp}
\end{equation}
\end{enumerate}

\subsection{Elliptic self-adjoint perturbations}
\label{ssec:perturbation}

Let $g_\eps:[0,1]\to\R$ be a positive smooth function such that 
\[
\lim_{\eps\to0}g_\eps(x)=g(x),\quad \text{pointwise},
\]
and consider the $\eps$-perturbed problem:
\begin{align}
\label{eq:epspert_1}
& \ueps_t=\left(g_\eps(x)\ueps\right)_{xx}-\left(g_\eps(x)\psi(x)\ueps\right)_x\ ,\\
\label{eq:epspert_2}
&\frac{\rd\ }{\rd t}\int \ueps(x,t)\rd x=0\ ,\\
\label{eq:epspert_3}
&\frac{\rd\ }{\rd t}\int \varphi(x)\ueps(x,t)\rd x=0\ ,\\
\label{eq:epspert_4}
&\ueps(x,0)=u_\ini(x)\ .
\end{align}
when $g(1)=0$. If $g(1)>0$  we then replace \eqref{eq:epspert_3} by
\begin{equation}
\left.\partial_x(g_\eps(x)\ueps(1,t))-g_\eps(x)\psi(x)\ueps(x,t)\right|_{x=1}=0.
\label{eq:epspert_5}
\end{equation}

This problem is now amenable to be treated using the ideas developed  in Section~\ref{sec:reg_cpp}. In order to write the problem in self-adjoint form and to obtain boundary conditions that are independent of $\eps$ we introduce the following change of variables:
\[
\ueps(x,t)=\frac{\veps(x,t)}{g_\eps(x)}p(x),\quad p(x)=\exp\left(\int_0^x\psi(y)\,\rd y.\right).
\]
In this new variable, we can apply Theorem~\ref{thm:eqv_cpp} and then the corresponding formulation given in \eqref{eq:def_tcp_m} for equations  \eqref{eq:epspert_1}, \eqref{eq:epspert_2}, \eqref{eq:epspert_3}, and \eqref{eq:epspert_4} becomes, respectively,
\begin{align}
\label{eq:mainmod}
\veps_t&=\frac{g_{\eps}}{p}\left(p\veps_x\right)_x\\
\label{eq:boundcondmod_1}
\veps_x(1,t)p(1)&=\veps_x(0,t)\\
\label{eq:boundcondmod_2}
\veps_x(1,t)p(1)-\varphi'(1)\veps(1,t)p(1)&=-\varphi'(0)v(0,t)\\
\label{eq:inicondmod}
\veps_\ini&\bydef\veps(x,0)= u_\ini\frac{g_{\eps}}{p}.
\end{align}
provided $g(1)=0$. If $g(1)>0$, we replaced \eqref{eq:boundcondmod_2} by
\begin{equation}
\partial_x\veps(1,t)=0
\label{eq:boundcondmod_3}.
\end{equation}
Theorem~\ref{thm:eup} applied to \eqref{eq:mainmod}, \eqref{eq:boundcondmod_1},\eqref{eq:boundcondmod_2}  and \eqref{eq:inicondmod} or applying Theorem~\ref{thm:pcp_eup} with \eqref{eq:boundcondmod_2} replaced by \eqref{eq:boundcondmod_3} yields the following result:
\begin{prop}
\label{prop:exist}
Let $u_\ini\in\bmp([0,1])$ and assume that $\psi \in L^1((0,1),\rd\mu)$. Then we have that \eqref{eq:mainmod}, \eqref{eq:boundcondmod_1} and \eqref{eq:boundcondmod_2} or \eqref{eq:boundcondmod_3})  has a unique solution in the class $C([0,\infty);\bmp([0,1]))\cap C^1([0,\infty);H^1((0,1)))$. Moreover, this solution can be written as
\[
\veps=\sum_{k=0}^\infty a_k\e^{-\lambda_kt}w_k,\quad a_k=(\veps_\ini,w_k).
\]
In addition, if $\veps_\ini>0$, then $\veps(\cdot,t)>0$.
\end{prop}

\begin{rmk}
\label{rmk:extra_reg}
If $\psi$ is continuous then  the eigenfunctions $w_k$ are $C^2$, and hence we have that $\veps$ is a $C^2$  classical solution for $t>0$.
\end{rmk}

\subsection{The vanishing perturbation limit}

\label{ssec:limits}

Since $\veps$  are positive we have also that  $\ueps$ are  positive, and because of  conservation law \eqref{eqn:cons_law_prob_genfp}  they can be seen as Radon measures  with a fixed given mass. In this case,   Prohorov's theorem --- cf. \cite{Billingsley:1999} --- implies  we have a (subsequencewise) limit as $\eps$ goes to zero. Namely,
\begin{prop}
Fix $T>0$ and let $\mathcal{C}(T)=[0,1]\times[0,T]$. Then, by passing a subsequence if necessary, we can assume that $\ueps\to u^{(0)}$ such that $u^{(0)}\in\bmp(\mathcal{C}(T))$.
\end{prop}

\begin{rmk}
First is important to observe that  we cannot interchange   the limits $T\to\infty$ and $\eps\to0$. Indeed, from the solution in the proof of Proposition~\ref{prop:exist} we have that  $\lim_{t\to\infty}\ueps(\cdot,t)$ is nonzero, regular and independent of $\eps$, since $1$ and $\varphi$ are independent of $\eps$. On the other hand, we shall see in Section~\ref{sec:prop} that the large time limit of $u^{(0)}$ is an atomic measure supported at the boundaries.
\end{rmk}
We now obtain an weak equation in weak form for the limiting measure. Let
\[
\Gamma\bydef C^1_c([0,\infty);\mathcal{D}),
\]
where 
\[
\D=\left\{\eta \in C^2((0,1))\cap C^1([0,1])\text{ such that \eqref{eq:boundcondmod_1} and \eqref{eq:boundcondmod_2} are satisfied}\right\}.
\]
The corresponding weakest formulation of \eqref{eq:mainmod} is given by
\begin{align*}
-\int_0^\infty\int_0^1 v^{(\eps)}(x,t)\frac{p(x)}{g_\eps(x)}\partial_t\alpha(x,t)\,\rd x\,\rd t - \int_0^1v^{(\eps)}(x,0)\frac{p(x)}{g_\eps(x)}\alpha(x,0)\,\rd x = \int_0^\infty\int_0^1v^{(\eps)}(x)\partial_x\left(p(x)\partial_x\alpha(x,t)\right)\rd x\,\rd t,
\end{align*}
with $\alpha\in\Gamma$.

Using the relationship between $u^{(\eps)}$ and $v^{(\eps)}$, the definition of $\rd\mu$ and that $p'=\psi p$ we obtain
\begin{align*}
-\int_0^\infty u^{(\eps)}(x,t)\partial_t\alpha(x,t)\,\rd x\,\rd t - \int_0^1u^{(\eps)}(x,0)\alpha(x,0)\,\rd x = \int_0^\infty\int_0^1u^{(\eps)}(x)g_\eps(x)\left(\partial^2_x\alpha(x,t) +\psi(x)\partial_x\alpha(x,t)\right)\rd x\,\rd t.
\end{align*}
Now, we consider the limit $\eps\to 0$:
\begin{prop}
\label{prop:conv}
The limiting measure $u^{(0)}$ is in the class $L^\infty([0,\infty);\bmp([0,1]))$, and it satisfies
\begin{align}
\nonumber
&-\int_0^\infty\int_0^1 u^{(0)}(x,t)\partial_t\alpha(x,t)\,\rd x\,\rd t - \int_0^1u^{(0)}(x,0)\alpha(x,0)\,\rd x\\
&\qquad= \int_0^\infty\int_0^1u^{(0)}(x,t)g(x)\left(\partial^2_x\alpha(x,t) +\psi(x)\partial_x\alpha(x,t)\right)\rd x\,\rd t,
\label{eq:wf_restr}
\end{align}
with test functions $\alpha\in\Gamma$.    In addition, it satisfies the conservation laws \eqref{eqn:cons_law_prob_genfp} and \eqref{eqn:cons_law_fix_genfp}, when $g(1)=0$ and the conservation law \eqref{eqn:cons_law_fix_genfp} and the boundary condition \eqref{eqn:bnd_cond_genfp}, when $g(1)>0$.
\end{prop}
\begin{proof}
Convergence follows from standard arguments; \cite{Billingsley:1999}. 
When $g(1)=0$, it remains only to show the conservation laws. This can be done either by appealing to standard convergence theorems and taking the limits in \eqref{eq:epspert_2} and \eqref{eq:epspert_3} or, if $\psi$ is at least continuous,  as in \cite{ChalubSouza09a}, by considering test functions of the form
\[
\alpha(x,t)=\beta(t)\gamma(x),\quad \beta(t)\in C^1_{c}((0,\infty)),\quad \gamma\in\mathrm{span}\{1,\varphi\}.
\]
When $g(1)>0$, the conservation law \eqref{eqn:cons_law_prob_genfp} is verified analogously. The boundary condition \eqref{eqn:bnd_cond_genfp} can be verified by  first observing that the solutions has to be smooth near $x=1$. This can be seen by considering test functions with compact support in $(1/2,1]\times(0,T]$, and then appealing by local parabolic regularity.  Integration by parts then yields the result.
\end{proof}

\begin{rmk}
Notice that regardless of the regularity of $\psi$, we always have that $1$ is in $\D$. On the other hand, we have $\varphi\in\D$ if, and only if, $\psi$ is continuous.
\end{rmk}

 \section{Further properties of the  weak solution}
 \label{sec:prop}

We now proceed to understand what properties the solution to \eqref{eq:wf_restr} possesses. In particular, when $g(1)=0$, we show uniqueness and that the solution found here is the same as the one found in \cite{ChalubSouza09b}
where the test space is taken to be $C^1_c([0,\infty);C^2([0,1]))$.

Before we can state our results,  we need a decomposition result  for compact distributions in \cite{ChalubSouza09b}, which we recall for the convenience of the reader. 

\begin{lem}[Decomposition]
\label{lem:decomp}
Denote by $\Ep$ the space of compactly supported distributions in $\R$. Let $\nu\in\Ep$ with $\mathop{\mathrm{sing\ supp}}(\nu)\subset [0,1]$. Then the setwise decomposition 
\[
[0,1]=\{0\}\cup(0,1)\cup\{1\}
\]
yields a decomposition in $\nu$, namely
\[
\nu=\nu_0+\mu+\nu_1,
\]
where $\nu_i$ is a compact distribution supported at $x=\{i\}$, and we also have that  $\mathop{\mathrm{sing\ supp}}(\mu)\subset(0,1)$. Moreover, if $\nu$ is a Radon measure, then $\mu\in\bm((0,1))$ and $\nu_i=c_i\delta_i$, where $\delta_i$ are normalised atomic measures with support in $x=\{i\}$. 
\end{lem}

We are now ready to discuss the two classes of examples that we are considering here.

\subsection{The generalised Kimura equation}
\label{ssec:gen_kimura}

The main result is then as follows

\begin{thm}
\label{thm:thm_prop}
Let $\Df$ be a domain such that  $C^2_c((0,1))\oplus\mathop{\mathrm{span}}(\{1\})\subset\Df$,  
and consider  \eqref{eq:wf_restr}, with test functions in $\Df$ and  initial condition   $u_\ini\in\bmp([0,1])$.  Applying the decomposition in Lemma~\ref{lem:decomp}, we can write
\begin{equation}
\label{eq:ic_dec}
u_\ini=a_0\delta_0 +r_0+b_0\delta_1,\quad  a_0,b_0,r_0\geq0.
\end{equation}
Then any solution to \eqref{eq:wf_restr} can   be represented as 
\[
u^{(0)}(\cdot,t)=a(t)\delta_0 + r + b(t)\delta_1,
\]
where $r$ is the unique  strong solution to  \eqref{eq:replicator_diffusion} without any boundary condition and initial condition $r_0$, and $a,b:[0,\infty)\to[0,\infty)$ are smooth. Furthermore,  we have that  
\begin{align*}
a(t)&=a_0+\int_0^1r_0(x)(1-\varphi(x))\,\rd x - \int_0^1r(x,t)(1-\varphi(x))\,\rd x\\
b(t)&=b_0+\int_0^1u_\ini(x)\varphi(x)\,\rd x - \int_0^1r(x,t)\varphi(x)\,\rd x.
\end{align*}
In particular, we have 
\[
u^{(0)}\in C^\infty((0,\infty); H^1)\cap C^0([0,\infty);\bmp),
\]
and that $u^{(0)}$ is unique in this class.
\end{thm}
 \begin{proof}
Notice that $C^2_c((0,1))\subset\mathfrak{D}$. Hence, by restricting to that domain we obtain the weak formulation to \eqref{eq:replicator_diffusion} without any boundary condition, and initial condition given by $r_0$. This equation is known to be well posed and to have a strong solution in $C^1((0,\infty);H^1(0,1))\cap C^0([0,\infty);\bmp)$. 

By applying the decomposition in $u^{(0)}$ together with \eqref{eq:wf_restr}  and still restricting to test functions in $C^2_c((0,1))$, we conclude that the non-atomic part of $u^{(0)}$ must be $r$.  

Since $u^{(0)}$ is a Radon measure, Lemma~\ref{lem:decomp} yields
\begin{equation}
\label{eq:decomp_uzero}
u^{(0)}(\cdot,t)=a(t)\delta_0+r(\cdot,t) + b(t)\delta_1.
\end{equation}
The formulas for $a$ and $b$ follow from direct substitution of \eqref{eq:decomp_uzero} in the integrated forms of \eqref{eq:conserv_law_1} and \eqref{eq:conserv_law_2}.
 \end{proof}
 
 \begin{rmk}
 In particular, we have that the solution is unique, and that it does not depend on the particular domain $\mathfrak{D}$---provided it satisfies the required condition. Notice that both $\mathcal{D}$ and $C^1_0([0,\infty);C^2([0,1])$ both satisfies this requirement, and hence the solution obtained from the elliptical perturbations is the same as the one obtained in \cite{ChalubSouza09b}. In particular, the asymptotic limit in time is given by
 \[
u^{(0)}_\infty(x)=a_\infty\delta_0 + b_\infty\delta_1,\quad (a_\infty,b_\infty)\bydef\lim_{t\to\infty}(a(t),b(t)).
 \]
 \end{rmk}

Further regularity in $\psi$ allows for a more detailed description of $a$ and $b$:

\begin{thm}
\label{thm:thm_prop2}
In the same framework of Theorem \ref{thm:thm_prop} assume, in addition,  that $\psi$ is continuous and that $\Df\supset C^2_c((0,1))\oplus\mathop{\mathrm{span}}(\{1,\varphi\})$. Then we have that
\[
u^{(0)}\in C^\infty((0,\infty);C^2([0,1]))\cap C^0([0,\infty);\bmp([0,1])).
\]
Moreover, we have that
\[
a(t)=a_0+\int_0^tr(0,s)\,\rd s
\quad\text{and}\quad
b(t)=b_0+\int_0^tr(1,s)\,\rd s.
\]
\end{thm} 
 
 \begin{proof}
 The extra regularity follows from the improved regularity of the eigenfunctions---cf. Remark~\ref{rmk:extra_reg}. 
 Then direct substitution of \eqref{eq:decomp_uzero} in \eqref{eq:wf_restr}, and on taking advantage of the improved regularity of $r$ to integrate by parts yields 
 \[
 \int_0^\infty a(t)\partial_t\alpha(0,t)\,\rd t + a_0\alpha(0,0)+b_0\alpha(1,0) + \int_0^\infty r(0,t)\alpha(0,t)\,\rd t + \int_0^\infty r(1,t)\alpha(1,t)\,\rd t.
 \]
 Choosing test functions that vanish at either endpoint then finishes the proof.
 \end{proof}

\subsection{The SIS-PDE model}
\label{ssec:sis}

Now, we return to the SIS-PDE model~(\ref{sis_pde})--(\ref{eqn:cons_law}).
Let
\[
F(x)=R_0(1-x)+1,\quad \text{and} \quad H(x)=x+\frac{2}{R_0	}\log\left(\frac{F(x)}{F(0)}\right).
\]
Also let
\[
\omega(x)=\frac{P(x)}{xF(x)},\quad P(x)=\exp(2H(x)),
\]
We shall also  write
\[
\omega^\epsilon=\frac{P(x)}{(x+\epsilon)F(x)}
\quad\text{and}\quad
\veps(x,t)=\frac{p(x,t)}{\omega^\epsilon(x)}
\]
Then, replacing the conservation law by the corresponding non-local boundary condition,  we obtain
\begin{align}
&\partial_t\veps=\frac{1}{2\omega^\epsilon(x)}\partial_x\left(P(x)\partial_x\veps\right)\nonumber\\
&P(1)\partial_x\veps(1,t)-P(0)\partial_x\veps(0,t)=0\nonumber\\
&\partial_x\veps(1,t)=0,\label{eqn:transf_regpart}\\
&\veps(x,0)=\frac{1}{\omega^{\epsilon}(x)}p_\ini(x)\nonumber
\end{align}
Theorem 3 then immediately translates into the following result:
\begin{thm}
For each $\epsilon>0$, Equation~\eqref{eqn:transf_regpart} has a unique solution, that is positive (non-negative) if $p_\ini$ is positive (non-negative). We also have that $\veps\to v^{(0)}$ weakly in  $[0,1]$, as $\epsilon\to0$. Let $p^{(0)}$ be the  limiting solution in the original variables. Then it satisfies:
\begin{align}
\nonumber
&\int_0^\infty\int_0^1p^{(0)}(x,t)\partial_t\phi(x,t)\rd x\rd t\\
\nonumber
&\quad+ \frac{1}{2}\int_0^\infty\int_0^1p^{(0)}(x,t)x\left(R_0(1-x)+1\right)\partial_x^2\phi(x,t)\rd x\rd t\\
\nonumber
&\quad+\int_0^\infty\int_0^1p^{(0)}(x,t)x\left(R_0(1-x)-1\right)\partial_x\phi(x,t)\rd x\rd t\\
&\quad+\int_0^1p^{(0)}(x,0)\phi(x,0)\rd x=0.
\label{eqn:weak_sis}
\end{align}
In particular it satisfies the conservation law~\eqref{eqn:cons_law}, and it can be written as
\begin{equation}\label{eqn:decomp_sol}
p^{(0)}(t,x)=a(t)\delta_0+r(x,t),\quad a(t)=\frac{R_0+1}{2}\int_0^tr(0,s)\,\rd s + a_0,
\end{equation}
and where $r$ satisfies equation~\eqref{sis_pde} with the boundary condition
\begin{equation}
\label{eqn:bc_reg_part}
\frac{1}{2}\left((1-R_0)r(1,t)+\partial_xr(1,t)\right)+r(1,t)=0
\end{equation}
\end{thm}

\section{Discussion}\label{sec:discussion}

This work is a first systematic step into the theory of conservative and degenerated parabolic problems. These problems appear in a number of modelling situations, as discussed in Section~\ref{sec:intro}. We restrict ourselves to 1+1 problems, and begin by presenting the theory for the case of  uniformly parabolic infinitesimal generators in
Section~\ref{sec:reg_cpp}, which  provides a comprehensive formulation for such problems from the point of view of Sturm-Liouville theory. As a by-product of the analysis, we also present a proof of persistence of positiveness that seems to be new even for the traditional Neumann problem for the Heat equation, since it does not require the study of the solution at the boundaries. The work in this section can be seen as an alternative approach to \cite{Mugnolo:Nicaise:2014} that naturally extends for equations with non-constant coefficients. They also show how these non-local problems can be brought into a similar framework than the more classical boundary value problems.

If the operator is degenerated at at least one of the endpoints then two situations can arise: if the operator is self-adjoint, the analysis goes through unchanged. Otherwise, although one can still bring the problem into self-adjoint form,  it is possible that the solution in the original variables will not be  integrable.  An example of this situation is the generalised Kimura equation~(\ref{eq:replicator_diffusion}), and this is studied in sections~\ref{ssec:perturbation} and \ref{ssec:gen_kimura}. Indeed, when working in self-adjoint variables for this problem, one needs to enforce homogeneous Dirichlet conditions which, in turn, yield a negative-definite problem --- hence no conservation seems possible. However, by a perturbation argument we can apply the results of Section~\ref{sec:reg_cpp}, and by considering positive solutions and the weakest formulation for the problem, we can take the limit of vanishing perturbation  in measure space, and obtain a measure solution for the original equation. This was discussed in section~\ref{ssec:limits}. 

In section~\ref{ssec:sis}, we  presented  an example of a recently derived  Fokker-Plank equation~(\ref{sis_pde}) that arises from an epidemiological problem, and that is half-degenerated, supplemented by one conservation law. The technique used to obtain measure solutions of this problem is similar to the one described previously in this section.

The degenerated examples presented already suggest that a  very natural extension of this work is to study in more detail the endpoint degenerated cases.  The aim would be to  provide a complete classification of these problems, when the infinitesimal generator is not self-adjoint. 

Another natural extension is to study the case in more than one spatial variable. The analysis in Section~\ref{sec:reg_cpp} made extensive use of 1-d Sturm-Liouville theory, the generalisation to more dimensions is far from straightforward. One example of interest is the generalised Kimura equation obtained in~\cite{ChalubSouza14} (more restricted examples appear in~\cite{Ewens,EpsteinMazzeo2013}). A second possible example of interest, that generalizes the SIS-PDE studied in this work, is the SIR-PDE (where a third class of individuals, the \textbf{R}emoved individuals is considered). Although the derivation is not complete, some preliminary results were presented in~\cite{ChalubSouza11}.

\section*{Acknowledgements}
OD: Supported by the Erasmus Mundus Action 2 Programme of the European Union, and Universidade Nova de Lisboa.

\noindent MOS:   Partially supported by CNPq under grant   \#~308113/2012-8 and acknowledge the hospitality of CMA/FCT/UNL.

\noindent FACCC: Partially supported by FCT/Portugal Strategic Project UID/MAT/00297/2013 (Centro de Matemática e Aplicações, Universidade Nova de Lisboa) and by a ``Investigador FCT'' grant.

\bibliographystyle{abbrv}
\bibliography{Degref.bib}

\end{document}